\documentclass[11 pt]{amsart}
\usepackage{amssymb}
\usepackage{amsmath}
\usepackage{amsfonts}
\usepackage{graphicx}
\usepackage{amsthm}
\usepackage{enumerate}
\usepackage[mathscr]{eucal}
\usepackage{verbatim}

 \theoremstyle{plain}
\newtheorem{theorem}{Theorem}[section]
\newtheorem{lemma}[theorem]{Lemma}
\newtheorem{proposition}[theorem]{Proposition}

\numberwithin{equation}{section}
\theoremstyle{plain}

\theoremstyle{remark}

\def\bbR{{\mathbb {R}}}
\def\bbE{{\mathbb {E}}}

\def\cS{\mathcal S}

\begin{document}

\date{August, 2008}

\title
{Lower bounds for dimensions of sums of sets}

\author[]
{Daniel M. Oberlin}

\address
{D. M.  Oberlin \\
Department of Mathematics \\ Florida State University \\
 Tallahassee, FL 32306}
\email{oberlin@math.fsu.edu}

%\subjclass{42B10, 42B99}
%\keywords{Fourier transforms of measures on curves,
%Fourier restriction problem}

\thanks{
This work was supported in part by NSF grant DMS-0552041.
}

\begin{abstract}
We study lower bounds for the Minkowski and Hausdorff dimensions of the algebraic 
sum $E+K$ of two sets $E,K\subset\bbR^d$.
\end{abstract}

\maketitle

\section{Introduction}
Suppose $E,K\subset\bbR^d$ are compact sets. We are interested in finding lower bounds for the Hausdorff and Minkowski dimensions of the 
sum set $E+K$ which are better than the trivial lower bound 
\begin{equation}\label{est.5}
\dim (E+K)\geq\max \big(\dim (E),\,\dim (K)\big).
\end{equation}
Our general approach will be to 
fix a \lq\lq nice" $K$ and to look for some function $\Phi \big(K,\dim (E)\big)$ such that 
\begin{equation}\label{est0}
\dim (E+K)\geq\Phi \big(K,\dim (E)\big)>\max \big(\dim (E),\,\dim (K)\big)
\end{equation}
for all Borel $E\subset\bbR^d$ with $\dim (E)<d$. (There are well-known results for \lq\lq generic"
sums, e.g. the fact, a consequence of \cite{M}, that if $E,K\subset\bbR$, then $$\dim (E+tK )\geq\min \{\dim (E )+\dim (K),1\}$$ for almost all $t\in\bbR$. And there are also some interesting results requiring special hypotheses on both $E$ and $K$. One is in \cite{PS}: 
suppose $a\in (0,1/2)$ and let $C_a$ be the Cantor set
$$
\big\{(1-a)\sum_{j=0}^{\infty}\omega _j a^j\,:\omega_j \in \{0,1\}\big\}.
$$
Then $\dim (C_a +C_b )=\min \{\dim (C_a )+\dim (C_b),1\}$ if $\log (b)/\log (a)$ is irrational. See also \cite{Sh}.)
 
Our observations will fall into two classes, depending on whether $\dim$ means Minkowski dimension or Hausdorff dimension. 

The case of (upper) Minkowski dimension $\dim_m$ is easier and the example of a curve $K$ in $\bbR^2$ 
appears to be typical here: if $K$ is a line segment then
one cannot improve the trivial bound; if $K$ is not a line segment then we will show that 
\begin{equation}\label{est1}
\dim_m (E+K) \geq  1+\frac{\dim_m (E)}{2}.
\end{equation}
It is easy to see that, without further hypotheses on the curve $K$, \eqref{est1} cannot be improved: let $E$ be the cartesian product of two copies of some 
Cantor set and let $K\doteq K_0$ be a curve consisting of a horizontal line segment followed by a vertical one. 

Our result for Minkowski dimension will be a generalization of \eqref{est1}:
we will observe that the lower bound
\begin{equation}\label{est2}
\dim_m (E+K) \geq  \dim_m (K)+\dim_m (E) -\frac{\dim_m (K) \dim_m (E)}{d}
\end{equation}
holds for certain Cantor sets $K\subset\bbR$ and 
also whenever $K$ is a $k$-dimensional manifold in $\bbR^d$ which satisfies a certain nondegeneracy condition. 

For Hausdorff dimension $\dim_h$ we know very little. We will make the trivial observation that if $K\subset\bbR^d$ is a Salem set, then there is
the optimal estimate $\dim_h (E+K) =\min\{\dim_h (E)+\dim_h (K),d\}$. We will  show that if $K$ is the \lq\lq middle thirds" Cantor set in $\bbR$, then 
\begin{equation}\label{est3}
\dim_h  (E+K) \geq \frac{1+\dim_h (E)}{2}
\end{equation}
(which improves the trivial bound only if $\dim_h (E)>2\,{\log 2}/{\log 3}-1$).
We will prove a result concerning convolution estimates which yields 
the analog of \eqref{est2} for a few model surfaces. And we will conclude with some remarks about nondegenerate curves.

We will consider Minkowski dimension in  \S\ref{Minkowski} and Hausdorff dimension in  \S\ref{Hausdorff}.

\section{Minkowski dimension}\label{Minkowski}

If $E\subset\bbR^d$ then $\dim_m (E)\geq \beta$ is equivalent to
\begin{equation}\label{est3.5}
m_d \big(E+B(0,\delta)\big)\geq C(\epsilon )\,\delta^{d-\beta+\epsilon}
\end{equation}
for $\epsilon >0$ and small $\delta >0$. Our strategy for proving \eqref{est2} depends on the observation that \eqref{est2} is implied 
by the estimate, to hold for all for Borel $F\subset \bbR^d$, 
\begin{equation}\label{est4}
m_d (F+K)\geq C(K)\,m_d (F)^{1-\alpha /d}
\end{equation}
where $\alpha=\dim_m (K)$: if \eqref{est4} holds and $\dim_m (E)\geq\beta$, then \eqref{est3.5} implies that 
$$
m_d \big(E+K+B(0,\delta )\big)\gtrsim \delta ^{d-(\alpha +\beta -\alpha\beta /d)+\epsilon '}
$$
where $\epsilon '=\epsilon(1-\alpha /d)$. Inequalities \eqref{est4} were the subject of the two papers \cite{O1} and \cite{O2}, and we will rely on 
results and ideas from those papers as we verify \eqref{est2} for various Cantor sets and $k$-surfaces. 

The results of \cite{O2} imply inequalities \eqref{est4} for certain Cantor sets $K$. To describe them we establish some notation. Fix a positive integer $n\geq 3$ and let $G(n)=\{0,1,\dots ,n-1\}$, which we will interpret as either a set of integers or as a realization of the group of integers modulo $n$. Fix a subset $S\subset G(n)$ such that $0\in S$ and consider the generalized Cantor set $K\subset [0,1]$ consisting of all sums $\sum_{j=1}^{\infty}s_j n^{-j}$ such that each $s_j \in S$. Then the Hausdorff and Minkowski dimensions of $K$ are both equal to $\log (|S|)/\log (n)$. Let $\widetilde{m}_n$ be normalized counting measure on $G(n)$ and $\widetilde{m}$ be Lebesgue measure on $[0,1)$ (which, when equipped with addition modulo $1$,  we regard here as a 
realization of the circle group). Theorem 2 of \cite{O2} states that if $\gamma\in (0,1)$ and the inequality
\begin{equation}\label{est5}
\widetilde{m}_n (E+S) \geq \widetilde{m}_n (E)^\gamma 
\end{equation} 
holds for $E\subset G(n)$, where the addition is in the group $G(n)$, then the inequality 
\begin{equation*}
\widetilde{m}(E+K)\geq \widetilde{m}(E)^\gamma 
\end{equation*} 
holds for $E\subset [0,1)$, with addition modulo $1$. Unwrapping the addition modulo $1$, this implies 
\begin{equation*}
m_1 (F+K)\geq m_1 (F)^\gamma /2 
\end{equation*}
for Borel $F\subset \bbR$. It is easy to verify that \eqref{est5} holds if $|S|=n-1$ and $\gamma =1-\log (|S|)/\log (n)$ (the case $n=3$ yielding the classical Cantor set). Thus, for these Cantor sets
$K$, \eqref{est4} and so \eqref{est2} hold.
There are also a few other cases where \eqref{est4} and \eqref{est2} are true, $n=5$ and $|S|=2$ giving one example. But the method of \cite{O2} is not very flexible 
and, in particular, does not seem to be up to solving the interesting problem of establishing \eqref{est4} for all of the Cantor sets $C_a$ mentioned in the introduction.

We now consider the situation when $K$ is a $k$-surface in $\bbR^d$. The simplest case is when $K$ is a curve in $\bbR^2$. As previously mentioned,
if $K$ is a line segment - which in this setting we consider to be a degenerate curve -  we cannot improve on the trivial bound \eqref{est.5}. If $K$ is not a line segment, an estimate of the form \eqref{est4} for $K$ follows from the more general estimate (Theorem B in \cite{O1}):
\begin{equation}\label{est6}
\sqrt{m_d (K-K)m_d (E)}\leq m_d (K+E).
\end{equation}
To be specific, if the curve $K$ is nondegenerate, i.e., not a line segment, then $m_2 (K-K)>0$ and so \eqref{est6} gives \eqref{est4} with
$\alpha =1,\,d=2$. This, in turn, gives \eqref{est1}. 

A first step towards generalizing \eqref{est1} is to find an appropriate notion of
nondegeneracy for a $k$-surface $K$ in $\bbR^d$. We will refer to a mapping $\Psi :K^d \rightarrow (\bbR^d )^k$ of the form 
\begin{equation}\label{infl}
\Psi :
(x_1 ,\dots ,x_d )\mapsto (\sum_{j=1}^d \tau_j^1 x_j ,\dots , \sum_{j=1}^d \tau_j^k x_j ),\ \tau_j^k \in\{-1,1\}
\end{equation}
as an \textit {inflation map} (the word \lq\lq inflation" in this context comes from \cite{C}). For example, if $k=1,\,d=2$, then $\Psi (x_1 ,x_2 )=x_1 -x_2$
is an inflation map.  We will consider the $k$-surface $K\subset\bbR^d$ to be \textit{nondegenerate} if there is an inflation map $\Psi$ such that
$\Psi (K^d )$ has positive Lebesgue measure in $ (\bbR^d )^k$. We would like to prove that if $K$ is a nondegenerate $k$-surface in $\bbR^d$, then 
\eqref{est4} holds with $\alpha =k$, and so \eqref{est2} holds. Unfortunately, the proofs that we have are tied to particular inflation maps. Thus, in particular, we do not even know whether there is an analog of \eqref{est6} with $K+K$ in place of $K-K$. The current situation is most satisfactory 
when $k=d-1$: Proposition 4 in \cite{O1} shows that if 
$$
\{(x_1 -x_2 ,\dots ,x_1 -x_d ):\ x_j \in K\}
$$
has positive Lebesgue measure in  $ (\bbR^d )^{d-1}$, then \eqref{est4} holds with $\alpha =d-1$. Thus \eqref{est2} holds.

The results of \cite{O1} were phrased in terms of a particular inflation map for $k$-surfaces in $\bbR^d$:
$$
\Psi_0  :(x_1 ,\dots ,x_d )\mapsto (x_1 +\cdots +x_{l}-x_{l+1},\dots ,x_1 +\cdots +x_{l}-x_{d}),\ l=d-k.
$$
Unfortunately (and unfortunately unnoticed when \cite{O1} was written), unless $k=1$ or $k=d-1$ or $d\leq 4$, 
if $K$ is a reasonable $k$-surface in $\bbR^d$, then $\Psi_0$ can never map $K^d$ onto a set of positive 
measure in $ (\bbR^d )^k$: roughly, the term $x_1 +\cdots +x_l$ uses up $lk$ dimensions but can make a contribution of only $d$ dimensions 
to the range of $\Psi_0$, and $lk=(d-k)k\leq d$ forces $k=1$ or $k=d$ unless $d\leq 4$. To rule out such $\Psi$ we now add 
another condition to the definition of inflation map: suppose $\{e_1 ,\dots ,e_d\}$ is the usual basis for $\bbR^d$ and define $K_0 \subset\bbR^d$ 
by
\begin{equation}\label{K}
K_0 =\bigcup_{\{1\leq i_1 <\cdots  <i_k \leq d\}}
\Big\{ a_{i_1} e_{i_1}+\cdots +a_{i_k} e_{i_k}:-\frac{1}{2}\leq a_{i_j} \leq \frac{1}{2}\Big\}
\end{equation}
so that $K_0$ is $k$-dimensional and analogous to the curve $K_0$ mentioned after \eqref{est1}. 
Then an \textit{inflation map} is (re)defined to be a map $\Psi$ of the form \eqref{infl} for which $\Psi \big((K_0)^d \big)$ has positive Lebesgue measure in 
$(\bbR^d )^k$. Here is an example (and for the remainder of this paper, with $k$ and $d$ fixed, $\Psi$ will stand for this particular example).
Write $d=qk+r$ with $q,r$ nonnegative integers and $0\leq r<k$. Define $\Psi$ by
\begin{equation*}
\Psi  :x=(x_1 ,\dots ,x_d )\mapsto \big(\psi _1 (x),\dots ,\psi _k (x)\big) 
\end{equation*}
where, for $1\leq j\leq k$, we choose $n_j \in\{ 0,1,\dots ,r-1\}$ such that 
$$
n_j k<jr\leq (n_j +1)k
$$
and then set
\begin{equation*}
\psi_j (x) =(x_d +x_{d-1}+\cdots +x_{d-n_j})+(x_{jq} +x_{jq-1}+\cdots +x_{jq-(q-2)} -x_{jq-(q-1)}) .
\end{equation*}
For example, if $d=5$ and $k=2$, then
$$
\Psi (x_1 ,\dots ,x_5 )=(x_5 +x_2 -x_1 ,x_5 +x_4 -x_3 ).
$$
Also, define $\psi_j^{'}$ and $\psi_j^{''}$ by 
\begin{equation}\label{defpsi}
\psi_j^{'} (x) =x_{jq} +x_{jq-1}+\cdots +x_{jq-(q-2)} -x_{jq-(q-1)},\ 
\psi_j^{''} (x) =x_d +x_{d-1}+\cdots +x_{d-n_j}
\end{equation} 
and 
\begin{equation}\label{def2}
\Psi^{'}(x)=\big(\psi_1^{'}(x),\dots ,\psi_k^{'}(x)\big),\,\,\Psi^{''}(x)=\big(\psi_1^{''}(x),\dots ,\psi_k^{''}(x)\big).
\end{equation}

With a view towards Lemma \ref{lemma1} below, as well as to establish that $\Psi \big( (K_0 )^d \big)$ has positive Lebesgue measure in $(\bbR^d )^k$, we now indicate how to construct probability measures $\lambda _1 , \dots ,\lambda _d$ on $K_0$ such that 
\begin{equation}\label{ineq}
\int_{K_0}\cdots\int_{K_0}f\big(\Psi (x_1 ,\dots ,x_d )\big)\,d\lambda_1 (x_1 )\cdots d\lambda_d (x_d)\leq\int_{(\bbR^d)^k}f\, dm_{dk}
 \end{equation}
for nonnegative functions $f$ on $(\bbR^d )^k$.  Each of the $\lambda_j$'s will be $k$-dimensional Lebesgue measure on one of the sets 
$$
\Big\{ a_{i_1} e_{i_1}+\cdots +a_{i_k} e_{i_k}:-\frac{1}{2}\leq a_{i_j} \leq \frac{1}{2}\Big\}
$$ 
from \eqref{K}. 

To give the idea, we first treat the case $d=5$ and $k=2$ mentioned above. Choose $\lambda_1$ and $\lambda_2$ so that, if $g$ is a function on $\bbR^5$, then
$$
\int_{K_0}\int_{K_0} g(x_2 -x_1 )\,d\lambda_2 (x_2)\,d\lambda_1 (x_1 )=\int_{-1/2}^{1/2}\cdots\int_{-1/2}^{1/2}g(0 ,a_2^1 ,a_3^1 ,a_4^1 ,a_5^1 )\, da_2^1 \cdots da_5^1 
$$
and then $\lambda_3$ and $\lambda_4$ so that 
$$
\int_{K_0}\int_{K_0} g(x_4 -x_3 )\,d\lambda_4 (x_4)\,d\lambda_3 (x_3 )=\int_{-1/2}^{1/2}\cdots\int_{-1/2}^{1/2}g(a_1^2 ,0 ,a_3^2 ,a_4^2 ,a_5^2)\, da_1^2  \, da_3^2 
\, da_4^2 \, da_5^2.
$$
Then choose $\lambda_5$ so that 
$$
\int_{K_0} g(x_5 )\,d\lambda_5 (x_5)=\int_{-1/2}^{1/2}\int_{-1/2}^{1/2}g(b_{1,1},b_{1,2},0,0,0)\, db_{1,1} \, db_{1,2} .
$$
Clearly
$$
\int_{K_0}\cdots\int_{K_0} f(x_5 +x_2 -x_1 ,x_5 +x_4 -x_3 )\, d\lambda_1 (x_1 )\cdots d\lambda_5 (x_5 )\leq\int_{(\bbR^5)^2} f\,dm_{10}
$$
for nonnegative functions $f$ on $\bbR^5 \times \bbR^5$, giving \eqref{ineq}.

In the general case we write $(a_1^1 ,\dots ,a_d^1 ;a_1^2 ,\dots ,a_d^2 ;\dots ;a_1^k ,\dots ,a_d^k )$ for an element of $(\bbR^d )^k$ 
and split the variables 
$a_i^j$ into two classes. For a nonnegative integer $i$, let $[i]$ satisfy $1\leq [i]\leq d$ and $[i]=i$ mod $d$. For $1\leq j\leq k$ 
we will say that the $r$ variables $a_{[(j-1)r+1]}^{j},a_{[(j-1)r+2]}^j ,\dots ,a_{[jr]}^j$ are in the second class and the remaining $kq=d-r$ variables
$\{a_{p_{j,n}}^j\}_{n=1}^{kq}$ are in the first class.

Choose the first $kq$ of the measures $\lambda_i$ so that, for $1\leq j\leq k$ 
$$
\int_{(K_0)^q} g \big(\psi_j^{'}(x_{(j-1)q+1},
\dots ,x_{jq})\big)\,d\lambda_{jq}(x_{jq})\cdots d\lambda_{(j-1)q+1}(x_{(j-1)q+1})
$$ 
is equal to 
$$
\int_{-1/2}^{1/2}\cdots\int_{-1/2}^{1/2}g\big(\sum_{n=1}^{kq}a_{p_{j,n}}^j e_{p_{j,n}} \big)\,da_{p_{j,1}}^j \cdots da_{p_{j,qk}}^j 
$$ 
for functions $g$ on $\bbR^d$.

Now rename the sequence
$$
a_{[1]}^1,a_{[2]}^1,\dots ,a_{[r]}^1,a_{[r+1]}^2,\dots ,a_{[2r]}^2,\dots ,a_{[(k-1)r+1]}^k,\dots ,a_{[kr]}^k
$$
of variables in the second class as the sequence
$$
b_{{1,1}},b_{{1,2}}\dots ,b_{{1,k}},b_{{2,1}},\dots ,b_{{2,k}},\dots ,b_{{r,1}},\dots ,b_{{r,k}}.
$$
and similarly rename the sequence
$$
e_{[1]},e_{[2]},\dots ,e_{[r]},e_{[r+1]},\dots ,e_{[2r]},\dots ,e_{[(k-1)r+1]},\dots ,e_{[kr]}
$$
of unit vectors as
$$
e_{q_{1,1}},e_{q_{1,2}}\dots ,e_{q_{1,k}},e_{q_{2,1}},\dots ,e_{q_{2,k}},\dots ,e_{q_{r,1}},\dots ,e_{q_{r,k}}.
$$
Since $k<d$ it is clear that for each $j=1,2,\dots ,r$ the $k$ unit vectors $e_{q_{j,1}},e_{q_{j,2}}\dots ,e_{q_{j,k}}$ are distinct. For such $j$ define 
the measures $\lambda_{d-j+1}$ on $K_0$ by 
$$
\int_{K_0} g(x_{d-j+1})\,d\lambda_{d-j+1}(x_{d-j+1})=\int_{-1/2}^{1/2}\cdots\int_{-1/2}^{1/2}g\big(\sum_{n=1}^k
b_{{j,n}}e_{q_{j,n}}\big)\,db_{{j,1}}\cdots b_{{j,k}}.
$$
Let $\Pi_2$ represent the projection of $(\bbR^d)^k$ onto the $kr$-dimensional space corresponding to the variables of the 
second class and let $\Pi_1$ 
be the complementary projection. By the 
choice of the $n_j$'s it follows that, if $f$ is a nonnegative function on $(\bbR^d)^k$,
we have
\begin{equation*}
\int_{(K_0 )^r}f\big(\Pi_2 \circ \Psi^{''}(x_{d-r+1},\dots ,x_d )\big)\,d\lambda_{d-r+1}(x_{d-r+1})\cdots d\lambda_d (x_d )\leq
\int_{\Pi_2 \big((\bbR^d)^k\big)}f\,dm_{kr}
\end{equation*}
(to see this, write $(x_d ,\dots ,x_{d-r+1})=(b_{1,1},\dots ,b_{r,k})$ and observe that the matrix of the map
$$
(b_{1,1},\dots ,b_{r,k})\mapsto \Pi_2 \circ \Psi^{''}(x_{d-r+1},\dots ,x_d )
$$
is lower triangular with $1$'s on the diagonal).
Since 
\begin{equation*}
\int_{(K_0 )^{qk}}f\big(  \Psi^{'}(x_{1},\dots ,x_{d-r} )\big)\,d\lambda_{1}(x_{1})\cdots d\lambda_{d-r} (x_{d-r} )=
\int_{\Pi_1 \big([-1/2,1/2]^{dk}\big)}f\,dm_{k(d-r)}
\end{equation*}
it then follows that 
\begin{equation*}
\int_{(K_0 )^{d}}f\big(  \Psi^{'}(x_{1},\dots ,x_{d-r})+
\Psi^{''}(x_{d-r+1},\dots ,x_d )\big)
\,d\lambda_{1}(x_{1})\cdots d\lambda_{d} (x_{d} )\leq \int_{(\bbR^d)^k}f\,dm_{dk},
\end{equation*}
giving \eqref{ineq} as desired.

The main result of this section is the following theorem.
\begin{theorem}\label{thorem1}
Suppose $K$ is a $k$-dimensional $C^{(1)}$ surface in $\bbR^d$ which is nondegenerate in the sense that $m_{dk}\big(\Psi (K^d )\big)>0$. Then 
\eqref{est4} holds with $\alpha =k$ and so \eqref{est2} holds as well.
\end{theorem} 
\noindent The proof is an immediate consequence of the next two lemmas.
\begin{lemma}\label{lemma1}
Suppose $K\subset\bbR^d$ carries probability measures $\lambda_1 ,\dots \lambda_d$ such that 
\begin{equation}\label{ineq2}
\int_{K^d}f\big(\Psi (x_1 ,\dots ,x_d )\big)\,d\lambda_1 (x_1 )\cdots d\lambda_d (x_d)\lesssim \int_{(\bbR^d)^k}f\, dm_{dk}
 \end{equation}
 for nonnegative $f$ on $\bbR^d$. Then the estimate
 \begin{equation}\label{ineq3}
 \int_{\bbR^d}\prod_{j=1}^d \int_K \chi_E (y+x_j )\, d\lambda_j (x_j )\ dm_d (y) \lesssim m_d (E)^{d/(d-k)}
 \end{equation}
 holds for Borel $E\subset\bbR^d$.
\end{lemma}

\begin{lemma}\label{lemma2}
If $K$ is as in Theorem \ref{thorem1} then $K$ carries probability measures $\lambda_j$ such that \eqref{ineq2} holds.
\end{lemma}

\noindent To deduce \eqref{est4} with $\alpha =k$ from the lemmas, fix $F$ and take $E=K+F$ in Lemma \ref{lemma1}.

Here is the proof of Lemma \ref{lemma1}.
\begin{proof} The proof is based on ideas from \cite{C2} (see also \cite{C}).
Let $\Omega$ be defined by 
\begin{equation*}
\Omega = \int_{\bbR^d}\prod_{j=1}^d \int_K \chi_E (y+x_j )\, d\lambda_j (x_j )\ dm_d (y)
\end{equation*}
and put $\alpha =\Omega /m_d (E)$. It is enough to prove
\begin{equation}\label{ineq4}
\alpha ^{d-k} \lesssim m_d (E)^k .
\end{equation}
Since
\begin{equation*}
\Omega = \int_E\int_K \prod_{j=2}^d \Big(\int_K \chi_E (y+x_j -x_1 )\, d\lambda_j (x_j )\Big)\,
d\lambda_1 (x_1 )\ dm_d (y),
\end{equation*}
if $E_1$ is the set of $y\in E$ for which  
\begin{equation*}
\int_K \prod_{j=2}^d \Big(\int_K \chi_E (y+x_j -x_1 )\, d\lambda_j (x_j )\Big)\,d\lambda_1 (x_1 )\geq \frac{\alpha}{2},
\end{equation*}
then 
\begin{equation*}
 \int_{\bbR^d}\int_K \chi_{E_1} (y+x_1 )\, d\lambda_1 (x_1 )\prod_{j=2}^d \int_K \chi_E (y+x_j )\, d\lambda_j (x_j )\ dm_d (y)=
\end{equation*}
\begin{equation*}
\int_{E_1}\int_K \prod_{j=2}^d \Big(\int_K \chi_E (y+x_j -x_1 )\, d\lambda_j (x_j )\Big)\,
d\lambda_1 (x_1 )\ dm_d (y) \geq \frac{\Omega}{2}.
\end{equation*}
Continuing inductively, for $j=2,\dots d$ we produce nonempty sets $E_j \subset E$ such that 
\begin{equation}\label{ineq5}
\int_K  \Big(\prod_{p<j}\int_K \chi_{E_p}(y+x_j -x_p )\, d\lambda_p (x_p )  \prod_{q>j}\int_K \chi_{E}(y+x_j -x_q )\, d\lambda_q (x_q )\Big) d\lambda_j (x_j )
\geq \frac{\alpha}{2^j}
\end{equation}
for $y\in E_j$. 

We will make repeated use of the following observation (a consequence of \eqref{ineq5}): if $j=2,\dots ,d$ then 
\begin{equation}\label{ineq5.5}
\int_K \Big(\int_K \prod_{p=1}^{j-1}\chi_{E_{p}}(y +x_j -x_{p})\,d\lambda_p  (x_p )\Big)
d\lambda_{j}(x_{j})\gtrsim\alpha\chi_{E_j}(y).
\end{equation}

Here are some notational conventions which we will use in the remainder of the proof. The symbol $\bbE$ will denote 
an expectation and any subscripted $y$ will denote a random vector in $\bbR^d$. The underlying probability space will be the product
of a large number of copies of $K$ with probability measure a product of measures $\lambda_j$. Thus, with $y_p =-x_{p}$, 
we rewrite \eqref{ineq5.5} as
\begin{equation}\label{ineq5.6}
\bbE \Big(\int_K  
\prod_{p=1}^{j-1}\chi_{E_{p}}(y+y_p +x_j )
\,d\lambda_j (x_j )\Big)\gtrsim \alpha \,\chi_{E_j} (y).
\end{equation}
In particular, starting with $y\in E_d$ and then writing $y^1_p =y+y_p$, we have 
\begin{equation}\label{ineq5.75}
\bbE \Big(\int_K  
\prod_{p=1}^{d-1}\chi_{E_{p}}(y^1_p +x_d )
\,d\lambda_d (x_d )\Big)\gtrsim \alpha .
\end{equation}
Now, since 
\begin{equation*}
\bbE_{\{y^2_p\}} \Big(\int_K  
\prod_{p=1}^{d-2}\chi_{E_{p}}(y^1_{d-1}+y^2_{p}+x_d +x_{d-1} )
\,d\lambda_{d-1} (x_{d-1} )\Big)
\gtrsim
\alpha\,\chi_{E_{d-1}}(y^1_{d-1} +x_d )
\end{equation*}
by \eqref{ineq5.6}, we see, upon replacing $y^1_{d-1}+y^2_{p}$ by $y^2_{p}$,
%replacing the random vector $y$ by the random vector $y-x_{d-2}$ 
that \eqref{ineq5.75} yields
\begin{equation}\label{ineq5.8}
\bbE \Big(\int_K 
\int_K 
\prod_{p_1 =1}^{d-1}\chi_{E_{p_1}}(y^1_{p_1} +x_d )
\prod_{p_2 =1}^{d-2}\chi_{E_{p_2}}(y^2_{p_2}+x_d +x_{d-1} )
\,d\lambda_d  (x_d )\,d\lambda_{d-1}(x_{d-1})\Big)\gtrsim \alpha^2 .
\end{equation}
Next use \eqref{ineq5.6} again to write
\begin{equation*}
\bbE_{\{y_p\}} \Big(\int_K  
\prod_{p=1}^{d-3}\chi_{E_{p}}(y_{p}+y^2_{d-2}+x_d +x_{d-1} +x_{d-2})
\,d\lambda_{d-2} (x_{d-2} )\Big)
\gtrsim
\alpha\,\chi_{E_{d-2}}(y^2_{d-2} +x_d +x_{d-1})
\end{equation*}
and apply this to \eqref{ineq5.8} to obtain
\begin{equation*}
\bbE \Big(\int_K   \int_K  \int_K
\prod_{p=1}^{3}\ 
\Big[\prod_{q=1}^{d-p}
\chi_{E_{q}}(y_{p,q}+x_d +\cdots  +x_{d-p+1})\Big]
d\lambda_d  (x_d )\, d\lambda_{d-1} (x_{d-1} )\, d\lambda_{d-2} (x_{d-2} )
\Big)\gtrsim \alpha^{3}. 
\end{equation*}
After $r-3$ more steps (where we recall that $r$ is defined by $d=qk+r$), we have
\begin{equation}\label{ineq6}
\bbE \Big(\int_K  \cdots \int_K  
\prod_{p=1}^{r}\ 
\Big[\prod_{q=1}^{d-p}
\chi_{E_{q}}(y_{p,q}+x_d +\cdots  +x_{d-p+1})\Big]
d\lambda_d  (x_d )\cdots d\lambda_{d-r+1}(x_{d-r+1})\Big)\gtrsim \alpha^{r}   .
\end{equation}

We make another notational convention: $\int \cdots d\lambda (x)$ will stand for an integral over a product of copies of
$K$ with respect to a product of measures $\lambda_j$ where the measures occurring in the product correspond to the 
variables $x_j$ appearing in the integrand. In particular we rewrite \eqref{ineq6} as 
\begin{equation}\label{ineq6.5}
\bbE \Big(\int
\prod_{p=1}^{r}
\Big[\prod_{q=1}^{d-p}
\chi_{E_{q}}(y_{p,q}+x_d +\cdots  +x_{d-p+1})\Big]
d\lambda (x)\Big)\gtrsim \alpha^{r}   .
\end{equation}

Now fix $j\in\{1 ,\dots ,k\}$. With $n_j$ as specified after \eqref{K}, \eqref{ineq5.6} gives 
%we have
%$d-n_j -1 \geq d-r =kq\geq jq$. Thus, as before,  

\begin{equation}\label{ineq7}
\bbE\Big(\int_K \chi_{E_{jq-1}}(y_j +x_d +\cdots +x_{d-n_j}+x_{jq})\,d\lambda_{jq} (x_{jq} )
\Big)\gtrsim \alpha\, \chi_{E_{jq}}(y_{n_j+1,jq}\,+x_d +\cdots +x_{d-n_j })
\end{equation}
for some random vector $y_j$.
Since $n_j <r$ implies $jq\leq kq=d-r<d-n_j$, the $k$ estimates \eqref{ineq7} can be applied in \eqref{ineq6.5} to give 
\begin{equation}\label{ineq8}
\bbE \Big(\int
\prod_{j=1}^k
\chi_{E_{jq-1}}(y_j +x_d +\cdots +x_{d-n_j}+x_{jq})\,
d\lambda (x)\Big)\gtrsim \alpha^{r+k}   .
\end{equation}
Analogous to \eqref{ineq7} there are, for $1\leq j\leq k$, the estimates 
\begin{equation*}
\bbE\Big(\int_K \chi_{E_{jq-2}}(y_j^{'} +x_d +\cdots +x_{d-n_j}+x_{jq}+x_{jq-1})\,d\lambda_{jq-1} (x_{jq-1} )
\Big)
\end{equation*}
\begin{equation*}
\gtrsim \alpha\, \chi_{E_{jq-1}}(y_{j}+x_d +\cdots +x_{d-n_j }+x_{jq}).
\end{equation*}
Using these estimates in \eqref{ineq8} gives
\begin{equation*}
\bbE \Big(\int
\prod_{j=1}^k 
\chi_{E_{jq-2}}(y_j^{'} +x_d +\cdots +x_{d-n_j}+x_{jq}+x_{jq-1})\,
d\lambda (x)\Big)\gtrsim \alpha^{r+2k}   .
\end{equation*}
Doing this $q-4$ more times we obtain
\begin{equation}\label{ineq9}
\bbE \Big(\int
\prod_{j=1}^k 
\chi_{E_{(j-1)q+2}}(y_j +x_d +\cdots +x_{d-n_j}+x_{jq}+x_{jq-1}+\cdots +x_{(j-1)q+3})\,
d\lambda (x)\Big)\gtrsim \alpha^{r+(q-2)k}   .
\end{equation}
Recall from \eqref{ineq5.5} that 
\begin{equation*}
\int_K \int_K \chi_{E_{p-1}}(y +x_p -x_{p-1})\,d\lambda_p  (x_p )
\,d\lambda_{p-1}(x_{p-1}) \gtrsim\alpha \, \chi_{E_p}(y)
\end{equation*}
and apply this with $p=(j-1)q+2$ to obtain the estimates 
\begin{equation*}
\int_K \int_K \chi_{E_{(j-1)q+1}}(y_j +x_d +\cdots +x_{d-n_j}+x_{jq}+\cdots +x_{(j-1)q+2}-x_{(j-1)q+1})\cdot
\end{equation*}
\begin{equation*}
d\lambda_{(j-1)q+2} (x_{(j-1)q+2}) \,d\lambda_{(j-1)q+1}(x_{(j-1)q+1})
\end{equation*}
\begin{equation*}
\gtrsim \alpha \chi_{E_{(j-1)q+2}}(y_j +x_d +\cdots +x_{d-n_j}+x_{jq}+\cdots +x_{(j-1)q+3}).
\end{equation*}
Using these in \eqref{ineq9} and recalling the definition of $\psi_j$ gives 
\begin{equation*}
\bbE \Big(\int
\prod_{j=1}^k 
\chi_{E_{(j-1)q+1}}\big(y_j +\psi_j (x_1 ,\dots ,x_d )\big)\,
d\lambda (x)\Big)\gtrsim \alpha^{r+(q-1)k}   
\end{equation*}
and so, since $E_i \subset E$ and $r+(q-1)k=d-k$, 
\begin{equation*}
\bbE \Big(\int
\prod_{j=1}^k 
\chi_{E}\big(y_j +\psi_j (x_1 ,\dots ,x_d )\big)\,
d\lambda (x)\Big)\gtrsim \alpha^{d-k}.   
\end{equation*}
Now applying the hypothesis \eqref{ineq2} with $f$ the indicator function of 
$$
\prod_{j=1}^k (E-y_j )\subset
(\bbR^d )^k
$$
yields \eqref{ineq4}, completing the proof of Lemma \ref{lemma1}.

\end{proof}

The proof of Lemma \ref{lemma2} is the same as the analogous part of the proof of
Theorem 3 in \cite{O2}, but for the sake of completeness (and since it is short), we will 
sketch the argument: parametrize $K$ by a $C^{(1)}$ map $\phi : (0,1)^k \rightarrow \bbR^d$. 
Let $\Phi : \big( (0,1)^k \big)^d \rightarrow (\bbR^d )^k$ be defined by
\begin{equation*}
\Phi (x_1 ,\dots ,x_d )=
\Psi \big(\phi (x_1 ),\dots ,\phi (x_d)
\big).
\end{equation*}
The hypothesis is that $\Phi \big( \big( (0,1)^k \big)^d \big)$ has positive Lebesgue measure in $ (\bbR^d )^k$.
It follows from Sard's theorem, continuity, and the inverse function theorem, that there are $\delta >0$ 
and nonempty open sets $O_j \subset (0,1)^k$ such that 
\begin{equation*}
|\det \Phi ^{'} (x_1 ,\dots ,x_d )|\geq \delta
\end{equation*}
on $\prod_{j=1}^d O_j$ and such that $\Phi$ is one-to-one on $\prod_{j=1}^d O_j$. If the measures $\lambda_j$ are defined by
\begin{equation*}
\int_K g\, d\lambda_j =
\frac{1}{m_k (O_j )}\int_{O_j} g\big(\phi (x)\big)\, dm_k (x),
\end{equation*}
then it is easy to see that the $\lambda_j$'s satisfy the conclusion of Lemma \ref{lemma2}.

\section{Hausdorff dimension}\label{Hausdorff}

One approach to studying $\dim_h (E+K)$ begins with the fact that if $K$ and $E$ are, respectively, the supports of measures 
$\lambda$ and $\mu$, then $K+E$ is the support of $\lambda\ast\mu$. An attempt to exploit this idea might start with
Frostman measures  on $K$ and $E$ and hope to say something useful about the energies of $\lambda\ast\mu$. In our 
context, with $K$ fixed and desiring to estimate $\dim_h (E+K)$ for general $E$, it seems necessary to require more than that
$\lambda$ be a Frostman measure for $K$. If we require much more, namely that $K$ be a Salem set, then it is easy 
to show that 
\begin{equation}\label{ineq9.5}
\dim_h (E+K)=\min\{\dim_h (E)+\dim_h (K), d\}.
\end{equation}

Recall that a set $K\subset\bbR^d$
satisfying $\dim_h (K)=\alpha$ is a Salem set if, for each $s<\alpha$, $K$ carries a probability measure $\lambda$ satisfying
\begin{equation}\label{ineq9.2}
|\widehat{\lambda}(\xi )|\lesssim |\xi |^{-s/2}.
\end{equation}
(Kahane's book \cite{JK} is a good source of information about Salem sets.)
Suppose that $K$ is a Salem set and also that $\dim_h (E)=\beta$.
Now $\dim_h (E)=\beta$ is equivalent to the statement that if $r<\beta$ then there
is a probability measure $\mu$ supported on $E$ such that 
\begin{equation}\label{ineq10}
\int_{\bbR^d}|\xi |^{r-d}|\widehat{\mu}(\xi )|^2 \, d\xi <\infty  .
\end{equation}
If \eqref{ineq9.2} holds, then \eqref{ineq10} gives
\begin{equation*}
\int_{\bbR^d}|\xi |^{r+s-d}|\widehat{\lambda\ast\mu}(\xi )|^2 \, d\xi <\infty .
\end{equation*}
Thus whenever $r<\beta$ and $s<\alpha$, $E+K$ carries a probability measure $\nu$ such that 
\begin{equation*}
\int_{\bbR^d}|\xi |^{r+s-d}|\widehat{\nu}(\xi )|^2 \, d\xi <\infty ,
\end{equation*}
and \eqref{ineq9.5} follows. Of course this argument shows that if $K$ carries a probability measure $\lambda$ satisfying
\eqref{ineq9.2} then 
\begin{equation}\label{ineq10.5}
\dim_h (E+K)\geq \min\{\dim_h (E)+s,d\}.
\end{equation}
Unfortunately,
the requirement that $K$ be a Salem set is stringent: for example, a $k$-surface in $\bbR^d$ can be a Salem set only if $k=d-1$. 

Here is a connection between the theory of $L^p \rightarrow L^q$ convolution estimates for nonnegative measures $\lambda$
and estimates for $\dim_h (E+K)$:
\begin{proposition}\label{prop1}
Suppose $\lambda$ is a probability measure on $K\subset\bbR^d$ which satisfies the convolution estimate 
\begin{equation}\label{ineq11}
\|\lambda\ast f \|_{L^q (\bbR^d )}\lesssim \| f\|_{L^p (\bbR^d )}
\end{equation}
for indices $1<p<q<\infty$. If $\dim_h (E)=\beta$ then $$\dim_h (K+E)\geq q^{'}\big[\frac{d}{p}-\frac{d}{q}+\frac{\beta}{p^{'}}\big].$$
\end{proposition} 
\noindent Before giving the proof we make some comments:

(i) If \eqref{ineq11} holds, then convolution with the characteristic 
function of a small ball shows that $(1/p,1/q)$ must lie in the triangle 
$\Delta (\alpha ,d)$ with vertices $(0,0)$, $(1,1)$, and $\big(d/(2d-\alpha),(d-\alpha)/(2d-\alpha)\big)$,
where $\alpha$ is the Hausdorff dimension of the support of $\lambda$. If $\lambda$ is supported 
on a $k$-surface in $\bbR^d$ and $k(k+3)<2d$, then \eqref{ineq11} implies additional necessary conditions which 
keep $(1/p,1/q)$ bounded away from $\big(d/(2d-\alpha),(d-\alpha)/(2d-\alpha)\big)$. But if, for example, 
$k>d/2$, then there are $k$-surfaces such that \eqref{ineq11} holds for all $(1/p,1/q)$ in the interior of the triangle $\Delta (k,d)$.
See \cite{O.5}.

(ii) If  \eqref{ineq11} holds for all $(1/p,1/q)$ in the interior of the triangle $\Delta (k,d)$, then 
it follows from Proposition \ref{prop1} that the analog
\begin{equation}\label{ineq12}
\dim_h (E+K) \geq  \dim_h (K)+\dim_h (E) -\frac{\dim_h (K) \dim_h (E)}{d}
\end{equation} 
of \eqref{est2} holds.

(iii) If \eqref{ineq11} holds, then 
\begin{equation*}
m_d (F)=\int_{\bbR^d} \lambda\ast\chi_F \, dm_d \leq m_d (F+K)^{1/q'}\|\lambda\ast\chi_F \|_{L^q (\bbR^d )}\lesssim 
m_d (F+K)^{1/q'}m_d (F)^{1/p}
\end{equation*}
yields 
\begin{equation}\label{ineq12.5}
m_d (F)^{q' /p'}\lesssim m_d (F+K),
\end{equation}
unless $q=1$.
In particular, if $(1/p,1/q)$ lies on the open segment joining $(1,1)$ and $\big(d/(2d-\alpha),(d-\alpha)/(2d-\alpha)\big)$, 
\eqref{ineq12.5} is \eqref{est4}. (We deduced inequalities of the form \eqref{est4} from Lemma \ref{lemma1}. If the measures $\lambda_j$ in Lemma \ref{lemma1} are all equal, then the conclusion of Lemma \ref{lemma1} is an estimate like \eqref{ineq11}.)

(iv) If $\lambda$ is the (middle thirds) Cantor-Lebesgue measure on $\bbR$, then \eqref{ineq11} holds for $(1/p,1/q)=(2/3,1/3)$.
Thus Proposition \ref{prop1} yields $\dim_h (E+K)\geq (\dim_h (E)+1)/2$ for the Cantor set $K$. This improves the trivial estimate 
for $\dim_h (E+K)$
only if $\dim_h (E)>2\log 2/\log 3-1$.

Here is the proof of Proposition \ref{prop1}: (The material through Lemma \ref{lemma3} is, for the reader's convenience, repeated from
\cite{O.75}.)
For $\rho >0$, let $K_\rho$ be the kernel defined on $\bbR ^d$ by $K_\rho (x)=|x|^{-\rho}\chi_{B(0,R)}(x)$ where 
$R=R(d)$ is positive.
Suppose that the finite nonnegative Borel measure $\nu$ is a $\gamma$-dimensional measure on $\bbR ^d$  
in the sense that $\nu \big( B(x,\delta )\big)\leq C(\nu )\, \delta ^{\gamma}$ for all $x\in\bbR ^d$ 
and $\delta >0$. If $\rho <\gamma$ it follows that 
\begin{equation*}
\nu\ast K_\rho \in L^{\infty}(\bbR ^d ).
\end{equation*}
Also 
\begin{equation*}
\nu\ast K_\rho \in L^{1}(\bbR ^d )
\end{equation*}
so long as $\rho <d$. Thus, for $\epsilon>0$, 
\begin{equation}\label{ineq8.4}
\nu\ast K_\rho \in L^{p}(\bbR ^d ), \ \rho=\gamma +\frac{1}{p}(d-\gamma)-\epsilon
\end{equation}
by interpolation. 
The following lemma is a weak converse of this observation.
\begin{lemma}\label{lemma3} If \eqref{ineq8.4} holds with $\epsilon =0$ and $p>1$, then 
$\nu$ is absolutely continuous with respect to Hausdorff measure of dimension
$\gamma -\epsilon$ for any $\epsilon >0$. Thus the support of $\nu$ has Hausdorff dimension 
at least $\gamma$.
\end{lemma} 

\begin{proof}

Recall from \cite{BL} (see p. 140) that, for $s\in\bbR$ and $1\leq p,q\leq\infty$, the norm $\|f\|^s_{p,q}$ of a distribution $f$ on $\bbR ^d$ in the Besov space $B^s_{p,q}$ can be defined by

$$
\|f\|^s_{pq}=\|\psi\ast f\|_{L^{p}(\bbR ^d )}+\Big(\sum_{k=1}^{\infty} \big(2^{sk}\,\|\phi _k \ast f\|_{L^{p}(\bbR ^d )}\big)^q \Big)^{1/q}
$$
for certain fixed  $\psi\in\cS (\bbR^d )$, $\phi \in C^{\infty}_c (\bbR^d )$, and where $\phi _k (x)=2^{kd}\phi (2^k x)$. If $\nu \ast K_\rho \in L^{p}(\bbR ^d )$, then $\|\nu \ast \chi_{B(0,\delta )} \|_{ L^{p}(\bbR ^d )}\lesssim \delta ^\rho$. It follows that 
$\|\nu\|^s_{pq}<\infty$ if $s<\rho -d=(\gamma -d)/p'$. Now, for $t>0$ and $1<p',q'<\infty$, the Besov capacity $A_{t,p',q'}(K)$ of a compact $K\subset \bbR ^d$ is defined by

$$
A_{t,p',q'}(K)=\inf \{\|f\|^t_{p',q'}:f\in C^{\infty}_c (\bbR^d ),\,f\geq \chi_K \}.
$$
It is shown in \cite{S} (see p. 277) that $A_{t,p',q'}(K)\lesssim H_{d-tp'}(K)$. Thus it follows from the duality of $B^s_{p,q}$ and $B^{-s}_{p',q'}$ that 

$$
\nu (K)\lesssim \|\nu\|^s_{pq}\, A_{-s,p',q'}(K)\lesssim   H_{d+sp'}(K)=H_{\gamma -\epsilon} (K)
$$
if $s=(\gamma -d-\epsilon)/p'$.

\end{proof}

\noindent Now to prove Proposition \ref{prop1}, assume that $E\subset\bbR^d$ satisfies $\dim_h (E)=\beta$. Then, if $\epsilon >0$, $E$ supports a probability measure
$\mu$ such that $\mu\ast K_\rho \in L^p (\bbR^d )$ where $\rho =\beta +(d-\beta )/p-\epsilon$. The hypothesis \eqref{ineq11}
implies that $\lambda\ast\mu\ast K_\rho \in  L^q (\bbR^d )$. Then, since $\mu\ast\lambda$ is supported on $E+K$, 
Lemma \ref{lemma3} implies that 
$$
\dim_h (E+K)\geq q^{'}\big[\frac{d}{p}-\frac{d}{q}+\frac{\alpha}{p^{'}}\big]-q^{'}\epsilon .
$$

We conclude this paper with a short discussion of the situation when $K$ is a curve in $\bbR^d$, and we begin with 
the case $d=2$. If $A\subset
\bbR$ satisfies $\dim_h (A\times A)>2\dim_h (A)$ then taking $K=K_0$, where $K_0$ is the polygonal curve mentioned after 
\eqref{est1}, and $E=A\times A$ shows that the analog for Hausdorff dimension  of \eqref{est1} can fail for curves which satisfy only the weak nondegeneracy requirement of \S \ref{Minkowski}. Things are much improved if we require that $K$ be 
a $C^{(2)}$ curve in $\bbR^2$ (and the situation is analogous for hypersurfaces in $\bbR^d$): then, 
if $K$ is not a line segment -- i.e., if $K$ is not degenerate in the context of the problem at hand -- the 
arclength measure $\lambda$ on an interval of $K$ where the curvature does not vanish will satisfy $|\widehat{\lambda}(\xi )|\lesssim |\xi |^{-1/2}$, 
$K$ will be a Salem set, and so we will have, by \eqref{ineq9.5},
\begin{equation}\label{ineq13}
\dim_h (E+K)\geq\min\{\dim_h (E)+1, 2\},\ E\subset\bbR^2 .
\end{equation}

An analog for $\bbR^3$ of \eqref{ineq13} is also not difficult to establish. But the argument depends on deep results from
 \cite{pramseeg} about the Sobolev mapping properties of averaging operators associated with certain curves in $\bbR^3$.
 Let $K=\{\gamma (t):0<t<1\}\subset\bbR^3$ be a curve $\{\gamma (t):0\leq t\leq1\}$ of finite type
(see \cite{pramseeg} for the definition) and let $\lambda$ be the measure induced on $K$ by setting $d\lambda =dt$. 
Suppose that $0<\beta\leq 2$ and that
$\mu$ is a Borel measure on $E$ which is $\beta$-dimensional in the sense that $\mu \big( B(x,\delta )\big)
\lesssim \delta^\beta$ for $\delta >0,\, x\in\bbR^3$. Fix $\epsilon>0$.
We will show that 
\begin{equation}\label{ineq26}
\delta^{-3}\langle \chi_{B(0,\delta )}\ast\mu\ast\lambda ,\chi_F \rangle
\lesssim \delta^{-[3-(1+\beta )+\epsilon ]/p}m_3 (F)^{1/p}
\end{equation}
for some $p$, all Borel $F\subset\bbR^3$, and all $\delta >0$. 
By a well known argument from \cite{B} it will then follow that
\begin{equation}\label{ineq25}
\dim_h (E+K)\geq\min\{\beta+1, 3\},\ E\subset\bbR^3 .
\end{equation}
With $\|\cdot\|_{p,s}$ denoting the norm of the $L^p$ Sobolev space 
$L^p_s (\bbR^3 )$, Theorem 1.1 in \cite{pramseeg} furnishes $p$ such that there is the convolution estimate
\begin{equation}\label{ineq27}
\| \lambda\ast f\|_{p,1/p}\lesssim \|f\|_{L^p (\bbR^3 )}.
\end{equation}
On the other hand, \eqref{ineq8.4} shows that 
\begin{equation*}
\mu\ast K_\rho \in L^{p'}(\bbR ^3 ) \ \, \text{if}\ \,  \rho -3=\frac{\beta -3}{p}-\epsilon
\end{equation*}
and so 
\begin{equation*}
\mu\ast K_\rho \in L^{p'}_{-1/p}(\bbR ^3 )  \ \, \text{if}\ \,  \rho -3=\frac{\beta -2}{p}-\epsilon .
\end{equation*}
Taking $f=\chi_F$ in \eqref{ineq27}
and $\rho =3+\frac{\beta -2}{p}-\epsilon$, this gives 
\begin{equation*}
\langle K_\rho \ast\mu\ast\lambda ,\chi_F \rangle
\lesssim m_3 (F)^{1/p}.
\end{equation*}
Since $\delta^{-\rho}\chi_{B(0,\delta )}\lesssim K_{\rho}$, \eqref{ineq26} follows.

One can hope that the analog of \eqref{ineq25} persists for nondegenerate curves in $\bbR^d$ when $d>3$. 
But there is currently no result like Theorem 1.1 in \cite{pramseeg} available in higher dimensions, and we can only 
make a few observations. If $d\lambda$ is $dt$ on a segment $K$ of 
the model curve $(t,t^2 ,\dots ,t^d )$ in $\bbR^d$, then the best estimate for $\widehat{\lambda}$ is $|\widehat{\lambda}(\xi )|\lesssim
|\xi |^{-1/d}$. Thus \eqref{ineq10.5} yields only  
\begin{equation*}
\dim_h (E+K)\geq\min\{\dim_h (E)+2/d, d\}.
\end{equation*}
And, using Christ's theorem from \cite{C2} about the $L^p \rightarrow L^q$ convolution properties of $\lambda$, Proposition \ref{prop1}
gives 
\begin{equation*}
\dim_h (E+K)\geq\min\{(1-1/d)\dim_h (E)+1, d\}.
\end{equation*}
On the other hand, given \eqref{ineq25} it is easy to see that if $E\subset\bbR^d$ satisfies $\dim_h (E)\leq 2$, then
\begin{equation*}
\dim_h (E+K)\geq\dim_h (E)+1.
\end{equation*}
This is because Marstrand's projection theorem \cite{M} implies that for almost all orthogonal projections $\pi$ of 
$\bbR^d$ onto a three-dimensional subspace we have $\dim_h \big(\pi (E)\big)=\dim_h (E)$. 
Choose such a $\pi$ and note that, by \eqref{ineq25} and the fact that $\pi (K)$ is of finite type,
\begin{equation*}
\dim_h (E+K)\geq \dim_h \big(\pi (E+K)\big)= \dim_h \big(\pi (E)+\pi (K)\big)\geq \dim_h \big(\pi (E)\big)+1.
\end{equation*} 
%Whether or not the inequality 
%\begin{equation*}
%\dim_h (E+K)\geq\min\{\dim_h (E)+1, d\}
%\end{equation*}
%holds if $1<\dim_h (E)<\frac{7}{3}$ seems to be an interesting problem.

\end{document}